\documentclass{article}%
\usepackage{amsmath}%
\usepackage{amsfonts}%
\usepackage{amssymb}%
\usepackage{graphicx}
\providecommand{\U}[1]{\protect\rule{.1in}{.1in}}
\newtheorem{theorem}{Theorem}

\newtheorem{lemma}[theorem]{Lemma}

\newenvironment{proof}[1][Proof]{\noindent\textbf{#1.} }{\ \rule{0.5em}{0.5em}}
\begin{document}
\pdfoutput=1

\title{Asymptotic Properties of Stieltjes Constants}
\author{Krzysztof Dominik Ma\'{s}lanka\\Polish Academy of Sciences\\Institute for the History of Science\\Nowy \'{S}wiat 72, 00-330 Warsaw, Poland\\e-mail krzysiek2357@gmail.com}
\maketitle

\begin{abstract}
We present a new asymptotic formula for the Stieltjes constants which is both
simpler and more accurate than several others published in the literature (see
e.g. \cite{Fekih-Ahmed}, \cite{Knessl Coffey}, \cite{Paris}). More
importantly, it is also a good starting point for a detailed analysis of some
surprising regularities in these important constants.

Keywords: Stieltjes constants, saddle point method, N\o rlund--Rice
integral\bigskip

\end{abstract}

\begin{flushright}
Mathematicians should look anew at old concepts

in solitude and in absolute, childlike innocence.

Alexandre Grothendieck (1928-2014)

\textit{R\'{e}coltes et Semailles} (unpublished text)\bigskip
\end{flushright}

\section{Introduction}

The Stieltjes constants $\gamma_{n}$ are essentially coefficients of the
Laurent series expansion of the Riemann zeta function around its only simple
pole at $s=1$:%
\begin{equation}
\zeta(s)=\frac{1}{s-1}+%
{\displaystyle\sum\limits_{n=0}^{\infty}}
\frac{\left(  -1\right)  ^{n}}{n!}\gamma_{n}\left(  s-1\right)  ^{n}%
\label{ZetaExpansion}%
\end{equation}

It is commonly believed that they are irrational numbers, and even
transcendental, however no rigorous proof of this has been given
\cite{Maslanka Wolf}. High precision numerical computations of them are quite
a challenge (see \cite{Maslanka Kolezynski} and references therein). A common
and frequently cited view is that "for large $n$, the Stieltjes constants grow
rapidly in absolute value, and change signs in a complex pattern"
\cite{Wikipedia}. The first view is beyond any doubt, as illustrated in the
Figure 1 below.%
\begin{figure}[ptb]%
\centering
\includegraphics[height=0.4\textheight]
{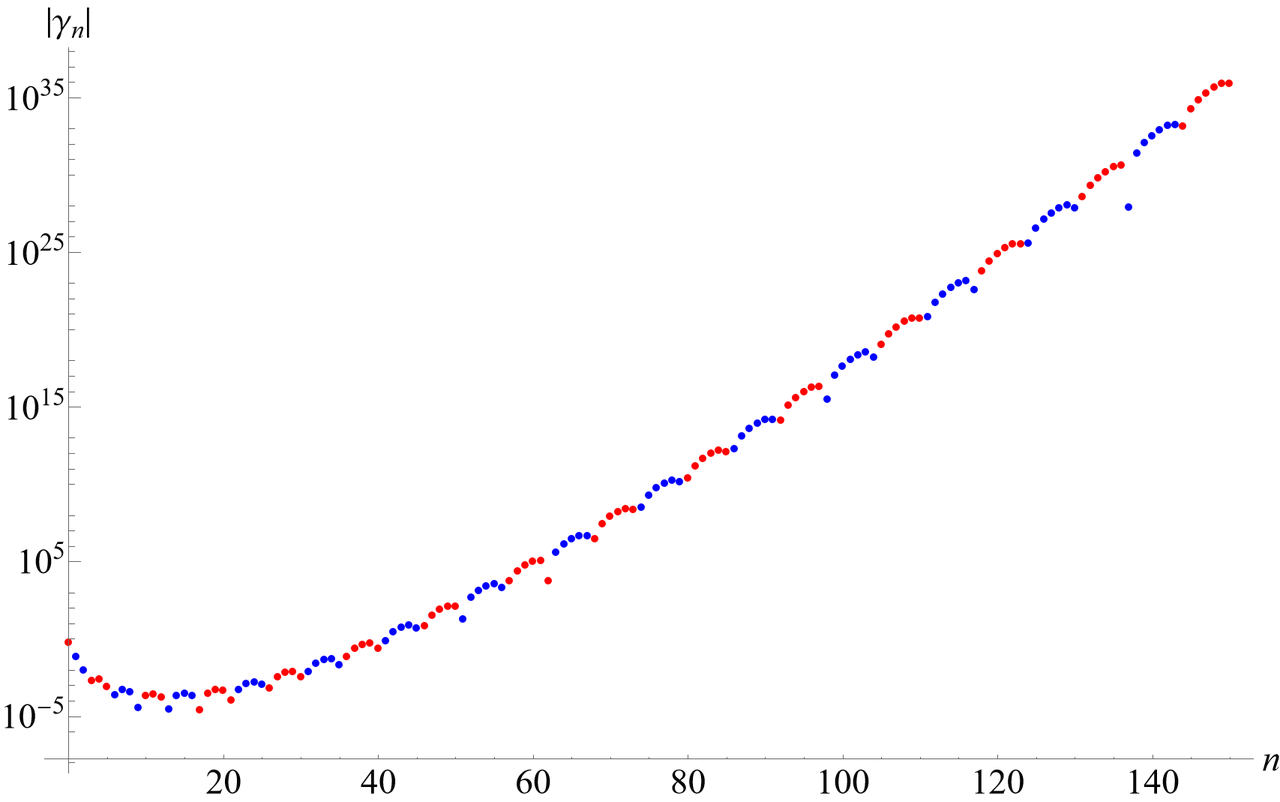}%
\caption{Absolute values of 150 initial Stieltjes coefficients $\gamma_{n}$.
Global, fast growing trend is evident. Oscillations of increasing amplitude
and decreasing frequency superimposed on this trend are visible. Red dots mean
positive values, blue dots mean negative values. The scale on the vertical
axis is logarithmic.}%
\end{figure}

In this paper, however, we will show that the second view is incorrect: not
only the signs of the Stieltjes constants, but their values also show amazing regularities.

There are three asymptotic formulas for these constants in the literature
(\cite{Knessl Coffey}, \cite{Fekih-Ahmed}, \cite{Paris}). We believe that the
one presented in this paper is definitely simpler than the others. It is also
more accurate. In particular, it recreates correctly the sign of $\gamma_{n}$
for the particular value of $n=137$ which is usually troublesome for
asymptotic formulas. Most importantly, this formula can be a starting point
for the analysis of the above-mentioned surprising regularities of Stieltjes
constants:%
\begin{equation}
\gamma_{n}\sim\sqrt{\frac{2}{\pi}}n!\operatorname{Re}\frac{\Gamma\left(
s_{n}\right)  e^{-cs_{n}}}{\left(  s_{n}\right)  ^{n}\sqrt{n+s_{n}+\frac{3}%
{2}}}\label{asymptotics 1}%
\end{equation}
where $s_{n}$ is the saddle point (see below):%
\begin{equation}
s_{n}=\frac{n+\frac{3}{2}}{W\left(  \frac{n+\frac{3}{2}}{2\pi i}\right)
}\label{Saddles 1}%
\end{equation}
In formula (\ref{asymptotics 1}) $c\equiv\ln\left(  2\pi i\right)  $ is a
complex constant and $W$ is the Lambert function (sometimes called the omega
function or product logarithm, see \cite{Wolfram Lambert}).

The basic tool is, as usual in such computations, the saddle point method
whereas the starting point is a certain alternating sum, which, due to the
still little known N\o rlund-Rice formula, can be converted into an integral
over the complex contour. As will be shown subsequently, global properties of
this integral clearly suggest using the saddle point method.

\section{Algorithm for calculating Stieltjes constants}

This work is a natural continuation of the previous one \cite{Maslanka
Kolezynski}. In that work, certain numerically efficient formula for Stieltjes
constants was given. In the present work, we will use this formula to derive a
new, effective formula for asymptotics for these important constants. As it
was done in \cite{Maslanka Kolezynski}, we will use polynomial interpolation
for the (regularized) Riemann zeta function $\varphi(s)$:%

\begin{equation}
\varphi(s):=\left\{
\genfrac{}{}{0pt}{}{\zeta(s)-\frac{1}{s-1}\qquad s\neq1}{\gamma\qquad
\qquad\qquad s=1}%
\right. \label{Regularized zeta}%
\end{equation}
where $\gamma$ is the Euler constants which stems from the appropriate limit.
In the mentioned interpolation, certain coefficients $\alpha_{k}$ appear
naturally, defined as follows:%
\begin{equation}
a_{k}(\varepsilon)=%
{\displaystyle\sum\limits_{j=0}^{k}}
\left(  -1\right)  ^{j}\binom{k}{j}\varphi(1+j\varepsilon)\label{a_k}%
\end{equation}
where $\varepsilon$ is certain real, not necessarily small number. (In what
follows we shall generally drop for simplicity this dependence in denotations:
$a_{k}(\varepsilon)\equiv a_{k}$.) Then, after some elementary computations,
we get:%
\begin{equation}
\fbox{$\gamma _{n}=\frac{n!}{\varepsilon ^{n}}\sum\limits_{k=n}^\infty
\frac{(-1)^{k}}{k!}\alpha _{k}S_{k}^{(n)}$}  \label{gamma gen}
\end{equation}
where $S_{k}^{(n)}$ are signed Stirling numbers of the first kind (see
\cite{Wolfram Stirling}). Formula (\ref{gamma gen}) is particularly well-suited
for numerical computations provided one has precomputed equidistant, high
precision values of $\varphi(s)$ in $s=1,1+\varepsilon,1+2\varepsilon,...$
(See paragraph IV of \cite{Maslanka Kolezynski} for all details.)

\section{Behavior of coefficients $a_{k}$}

Formula (\ref{a_k}) has special form of an alternating sum with binomial
coefficients. This form suggests using the N\o rlund--Rice integral which is a
powerful tool for dealing with such sums (see e.g. \cite{Flajolet Sedgewick}).

\begin{lemma}
Let $\varphi(s)$ be holomorphic in the half-plane $\Re(s)\geq n_{0}-\frac
{1}{2}$. Then the finite differences of the sequence $\{\varphi(k)\}$ admit
the integral representation:%
\begin{equation}%
{\displaystyle\sum\limits_{k=n_{0}}^{n}}
\left(  -1\right)  ^{k}\left(
\begin{array}
[c]{c}%
n\\
k
\end{array}
\right)  \varphi(k)=\frac{\left(  -1\right)  ^{n}}{2\pi i}\underset{C}{%
{\displaystyle\oint}
}\varphi(s)\frac{n!}{s(s-1)...(s-n)}\label{Noerlund-Rice}%
\end{equation}
where the contour of integration encircles the integers $\{n_{0},...,n\}$ in a
positive direction and is contained in $\Re(s)\geq n_{0}-\frac{1}{2}$.
\end{lemma}

\begin{proof}
According to the Cauchy residue theorem the contour integral on the right is
the sum of the residues of the integrand at $s=n_{0},...,n$ which is just
equal to sum on the left\footnote{Donald Knuth popularized this formula and
attributed it to American engineer Stephen O. Rice, pioneer in the
applications of probability techniques to engineering problems (1907-1986).
Knuth did it in one of the problem tasks at the end of one of the chapters of
his famous work\textit{\ }\cite{Knuth}. However, much earlier this formula was
known to Danish mathematician Niels Erik N\o rlund (1885-1981), who included
it in his extensive classic treatise \cite{Noerlund}. Incidentally, the
mentioned Niels Erik N\o rlund was the brother of Margrethe n\'{e}e Norlund,
later wife of the famous physicist Niels Bohr.}.
\end{proof}%

\begin{figure}[ptb]%
\centering
\includegraphics[height=0.12\textheight]
{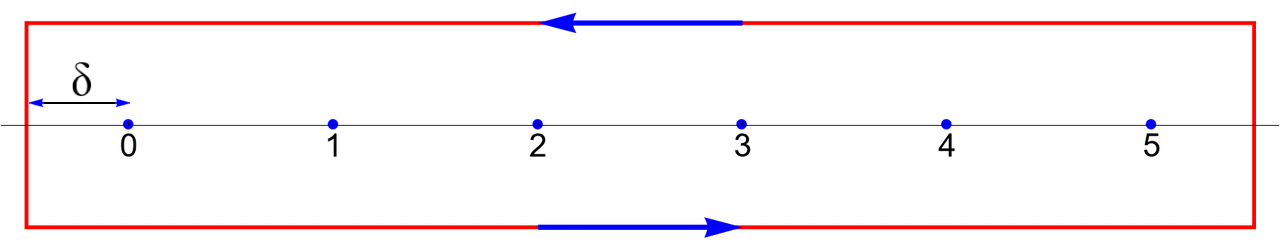}%
\caption{Rectangular contour $C$\ of integration for the right hand side of
(\ref{Noerlund-Rice}) for particular value $k=5$ encircling points
$0,1,...,5$. This shape is especially well suited for numerical
investigations.}%
\end{figure}

However, before applying the above Lemma it is convenient to make several
elementary transformations in (\ref{a_k}).%
\begin{align*}
a_{k}  & =\gamma+%
{\displaystyle\sum\limits_{j=1}^{k}}
\left(  -1\right)  ^{j}\binom{k}{j}\varphi(1+j\varepsilon)=\\
& =\gamma+%
{\displaystyle\sum\limits_{j=1}^{k}}
\left(  -1\right)  ^{j}\binom{k}{j}\zeta(1+j\varepsilon)-\frac{1}{\varepsilon}%
{\displaystyle\sum\limits_{j=1}^{k}}
\frac{\left(  -1\right)  ^{j}}{j}\binom{k}{j}%
\end{align*}
The last sum is%
\[%
{\displaystyle\sum\limits_{j=1}^{k}}
\frac{\left(  -1\right)  ^{j}}{j}\binom{k}{j}=-\gamma-\psi^{(0)}(k+1)=-H_{k}%
\]
where $\psi^{(0)}(s)$ is the polygamma function and $H_{k}\equiv\sum_{i=1}%
^{k}1/i$ is the $k^{\text{th}}$ harmonic number. Finally we get:%
\begin{equation}
a_{k}=\gamma+\frac{H_{k}}{\varepsilon}+%
{\displaystyle\sum\limits_{j=1}^{k}}
\left(  -1\right)  ^{j}\binom{k}{j}\zeta(1+j\varepsilon)\label{a_k zeta}%
\end{equation}

Now choosing the rectangular contour of integration (Figure 1) and applying
the above Lemma (\ref{Noerlund-Rice}) to (\ref{a_k}) we get:%

\begin{equation}
a_{k}=\frac{\left(  -1\right)  ^{k}k!}{2\pi i}\left(
{\displaystyle\int\limits_{-\delta+i\delta}^{-\delta-i\delta}}
f_{k}ds+%
{\displaystyle\int\limits_{-\delta-i\delta}^{k+\delta-i\delta}}
f_{k}ds+%
{\displaystyle\int\limits_{k+\delta-i\delta}^{k+\delta+i\delta}}
f_{k}ds+%
{\displaystyle\int\limits_{k+\delta+i\delta}^{-\delta+i\delta}}
f_{k}ds\right) \label{alpha_k integral 0}%
\end{equation}
where the integrand is:%

\[
f_{k}\equiv f_{k}(s,\varepsilon)=\frac{\varphi(1+s\varepsilon)}{%
{\displaystyle\prod\limits_{i=0}^{k}}
(s-i)}%
\]
and $\varphi$ is the regularized zeta function (\ref{Regularized zeta}) and
$\delta$ is positive parameter. (Typically $\delta=\frac{1}{2}$, see Fig. 2.)

Deforming the rectangular contour of integration to a vertical line
$\operatorname{Re}s=\frac{1}{2}$\ and a large semicircle on the right and
performing the integral along vertical line only, that is neglecting
contribution from the large semicircle, which tends to zero, we get:%

\[
a_{k}=\frac{\left(  -1\right)  ^{k}k!}{2\pi i}%
{\displaystyle\int\limits_{+\frac{1}{2}+i\infty}^{+\frac{1}{2}-i\infty}}
f_{k}(s,\varepsilon)ds
\]

After applying functional equation for the Riemann zeta function (see e.g.
\cite{Edwards}, p. 12-16)\footnote{Such a trick to use the functional equation
for the Riemann zeta function and then perform change of variable
$s\rightarrow-s$ was inspired by the work \cite{Flajolet Vepstas}, cf.
equations (17), (18) and the corresponding comment.}:%
\begin{equation}
\zeta(1+s\varepsilon)=\pi^{\frac{1}{2}+s\varepsilon}\frac{\Gamma
(-\frac{s\varepsilon}{2})}{\Gamma(\frac{1+s\varepsilon}{2})}\zeta
(-s\varepsilon)\label{Functional equation}%
\end{equation}
we get%

\[
a_{k}=\gamma+\frac{H_{k}}{\varepsilon}+\frac{\left(  -1\right)  ^{k}k!}{2\pi
i}%
{\displaystyle\int\limits_{+\frac{1}{2}+i\infty}^{+\frac{1}{2}-i\infty}}
\frac{\pi^{\frac{1}{2}+s\varepsilon}\frac{\Gamma(-\frac{s\varepsilon}{2}%
)}{\Gamma(\frac{1+s\varepsilon}{2})}\zeta(-s\varepsilon)}{%
{\displaystyle\prod\limits_{i=0}^{k}}
(s-i)}ds
\]

Performing change of variable $s\rightarrow-s$ yields:%

\[
a_{k}=\gamma+\frac{H_{k}}{\varepsilon}-\frac{\left(  -1\right)  ^{k}k!}{2\pi
i}%
{\displaystyle\int\limits_{-\frac{1}{2}-i\infty}^{-\frac{1}{2}+i\infty}}
\frac{\pi^{\frac{1}{2}-s\varepsilon}\frac{\Gamma(\frac{s\varepsilon}{2}%
)}{\Gamma(\frac{1-s\varepsilon}{2})}\zeta(s\varepsilon)}{%
{\displaystyle\prod\limits_{i=0}^{k}}
(-s-i)}ds
\]

Using elementary identity valid for integer $k$:%

\[%
{\displaystyle\prod\limits_{i=0}^{k}}
(-s-i)=-\left(  -1\right)  ^{k}%
{\displaystyle\prod\limits_{i=0}^{k}}
(s+i)
\]
and converting the product on the right into the Pochhammer symbol usually
denoted $(s)_{n}$:%

\[%
{\displaystyle\prod\limits_{i=0}^{k}}
(s+i)=\frac{\Gamma\left(  s+k+1\right)  }{\Gamma\left(  s\right)  }%
\equiv(s)_{k+1}%
\]
we get:%

\[
a_{k}=\gamma+\frac{H_{k}}{\varepsilon}+\frac{k!}{2\pi i}%
{\displaystyle\int\limits_{-\frac{1}{2}-i\infty}^{-\frac{1}{2}+i\infty}}
\pi^{\frac{1}{2}-s\varepsilon}\frac{\Gamma(\frac{s\varepsilon}{2})}%
{\Gamma(\frac{1-s\varepsilon}{2})}\frac{\Gamma\left(  s\right)  }%
{\Gamma\left(  s+k+1\right)  }\zeta(s\varepsilon)ds
\]
Now defining the integrand as:%

\begin{equation}
f_{k}(s,\varepsilon)=\pi^{\frac{1}{2}-s\varepsilon}\frac{\Gamma(\frac
{s\varepsilon}{2})}{\Gamma(\frac{1-s\varepsilon}{2})}\frac{\Gamma\left(
s\right)  }{\Gamma\left(  s+k+1\right)  }\zeta(s\varepsilon)\label{integrand}%
\end{equation}
we get%
\begin{align}
a_{k}  & =\gamma+\frac{H_{k}}{\varepsilon}+\frac{k!}{2\pi i}%
{\displaystyle\int\limits_{-\frac{1}{2}-i\infty}^{-\frac{1}{2}+i\infty}}
f_{k}(s,\varepsilon)ds=\label{a_k integral 1}\\
& =\gamma+\frac{H_{k}}{\varepsilon}+\frac{k!}{2\pi i}\left(
{\displaystyle\int\limits_{+\frac{1}{2}-i\infty}^{+\frac{1}{2}+i\infty}}
f_{k}(s,\varepsilon)ds-2\pi i\;\text{Res}\left(  f_{k}(s,\varepsilon
),0\right)  \right)
\end{align}
We can finally move the line of integration from $\operatorname{Re}s=-\frac
{1}{2}$ to $\operatorname{Re}s=+\frac{1}{2}$ and subtract the contribution
from residue of the integrand in $s=0$. It turns out that this residue is:%
\begin{equation}
\frac{\gamma\varepsilon+H_{k}}{\varepsilon k!}\label{Residue}%
\end{equation}
which miraculously cancels exactly the first and the second term in
(\ref{a_k integral 1})%
\begin{equation}
a_{k}=\frac{k!}{2\pi i}%
{\displaystyle\int\limits_{+\frac{1}{2}-i\infty}^{+\frac{1}{2}+i\infty}}
f_{k}(s,\varepsilon)ds\label{a_k integral 2}%
\end{equation}
It is convenient to introduce the following notation:%
\begin{align}
g_{k}(s,\varepsilon)  & \equiv\pi^{\frac{1}{2}-s\varepsilon}\frac{\Gamma
(\frac{s\varepsilon}{2})}{\Gamma(\frac{1-s\varepsilon}{2})}\frac{\Gamma\left(
s\right)  }{\Gamma\left(  s+k+1\right)  }\label{integrand 1}\\
f_{k}(s,\varepsilon)  & =g_{k}(s,\varepsilon)\zeta(s\varepsilon)\nonumber
\end{align}
%

\begin{figure}[ptb]%
\centering
\includegraphics[height=0.4\textheight]
{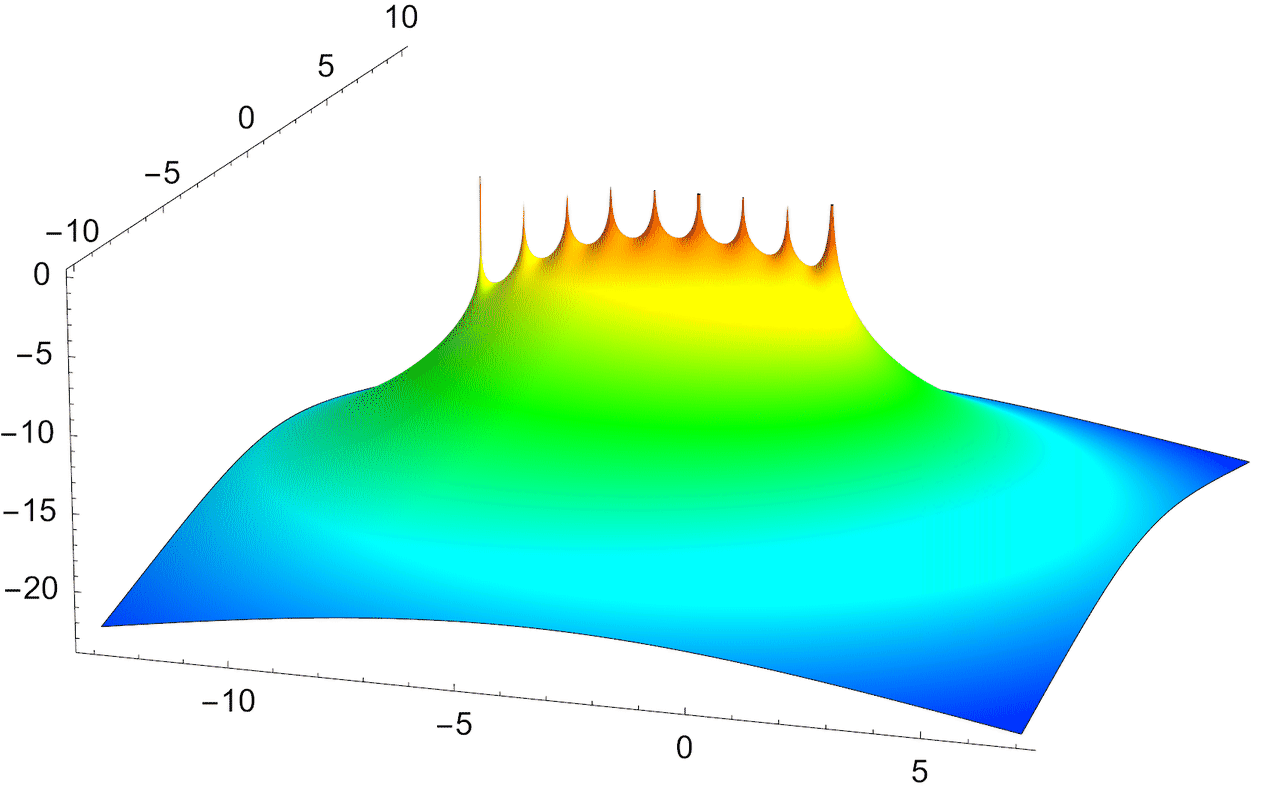}%
\caption{Absolute value of the integrand (\ref{integrand 1}) for $k=8$ and
$\varepsilon=2^{-5}$. Vertical scale is logarithmic for better visualisation.
The right half-plane of complex variable $s$ is free of singularities. Simple
poles in $s=0,-1,-2,...,-k$ due to factor $\Gamma(s)$ in (\ref{integrand 1})
are visible.}%
\end{figure}
The integrand in (\ref{integrand 1}) has several remarkable features. It is
free of singularities in the right half-plane and decays there exponentially
to zero. Hence, the vertical line of integration may be freely moved to the
right without any change of the integral. Therefore the integral is
well-suited for applying the saddle point method. Let us now remind the
following important result (see \cite{Erdelyi} for a very accessible
presentation of this method):

\begin{theorem}
The saddle-point method (or: Method of steepest descent). An integral
depending of some real parameter $\lambda$ may be approximated for large value
of this parameter as%
\begin{equation}
\int\tau(x)e^{\lambda\omega(x)}dx\sim\tau(x_{0})e^{\lambda\omega(x_{0})}%
\sqrt{-\frac{2\pi}{\lambda\omega^{\prime\prime}(x_{0})}},\quad\omega^{\prime
}(x_{0})=0\label{saddle point method}%
\end{equation}
(The solution $x_{0}$ of the equation $\omega^{\prime}(x_{0})=0$ is the saddle
point.)\footnote{Historical digression. We owe the original idea of this
method to Pierre Simon de Laplace (1774). Another contribution belongs to
Augustin Louis Cauchy (1829). In Bernhard Riemann's unpublished notes from
1863, this method is applied to hypergeometric functions. The final version
was published by Peter Debye (1909) who applied this method to Bessel
functions. Russian historians of mathematics recently reminded contribution of
Pavel Alexeevich Nekrasov, who (allegedly) discovered and used this method
independently a quarter of a century before Debye. I have no opinion on this
matter, since this Nekrasov was also a philosopher and used mathematics to
demonstrate the necessity of the tsarist regime and the need to maintain
secret services.}
\end{theorem}

In our case the discrete index $k$ plays the role of parameter $\lambda$
although it is not just multiplying factor. It is evident that in order to
apply the above theorem to integral (\ref{integrand 1}) one has to choose
$\tau\equiv1$ and $\omega=\log\left(  f_{k}(s,\varepsilon)\right)  $. More
precisely:%
\begin{equation}
\omega_{k}(s,\varepsilon)\equiv\log\left(  f_{k}(s,\varepsilon)\right)
\label{ln integrand}%
\end{equation}

All the computations below are elementary but very tedious, so they were
performed and checked with the help of Wolfram Mathematica \cite{Wolfram
Mathematica}.

We shall also need the first and the second derivative of the integrand
(\ref{integrand 1}) with respect to complex variable $s$. Having these we can
compute derivatives of $\omega_{k}(s,\varepsilon)$ as:%
\begin{equation}
\frac{\partial}{\partial s}\omega_{k}(s,\varepsilon)=\frac{\frac{\partial
}{\partial s}f_{k}(s,\varepsilon)}{f_{k}(s,\varepsilon)}\equiv\frac
{f_{k}^{(1)}(s,\varepsilon)}{f_{k}(s,\varepsilon)}\label{omega 1 der}%
\end{equation}%
\begin{equation}
\frac{\partial^{2}}{\partial s^{2}}\omega_{k}(s,\varepsilon)=\frac
{\frac{\partial^{2}}{\partial s^{2}}f_{k}(s,\varepsilon)}{f_{k}(s,\varepsilon
)}-\left(  \frac{\frac{\partial}{\partial s}f_{k}(s,\varepsilon)}%
{f_{k}(s,\varepsilon)}\right)  ^{2}\equiv\frac{f_{k}^{(2)}(s,\varepsilon
)}{f_{k}(s,\varepsilon)}-\left(  \frac{f_{k}^{(1)}(s,\varepsilon)}%
{f_{k}(s,\varepsilon)}\right)  ^{2}\label{omega 2 der}%
\end{equation}

Let $\psi(s)$ and $\psi^{(1)}(s)$ denote digamma function and its first
derivative, respectively. Introducing the following denotations:%
\[
p_{k}(s,\varepsilon)\equiv\psi(s)-\psi(s+k+1)+\frac{\varepsilon}{2}\left(
\psi\left(  \frac{s\varepsilon}{2}\right)  +\psi\left(  \frac{1-s\varepsilon
}{2}\right)  -2\log(\pi)\right)
\]%
\[
p_{k}^{(1)}(s,\varepsilon)\equiv\psi^{(1)}(s)-\psi^{(1)}(s+k+1)+\left(
\frac{\varepsilon}{2}\right)  ^{2}\left(  \psi^{(1)}\left(  \frac
{s\varepsilon}{2}\right)  -\psi^{(1)}\left(  \frac{1-s\varepsilon}{2}\right)
\right)
\]
after some elementary but tedious computations we get:%
\begin{equation}
f_{k}^{(1)}(s,\varepsilon)\equiv\frac{\partial}{\partial s}f_{k}%
(s,\varepsilon)=g_{k}(s,\varepsilon)\left(  p_{k}(s,\varepsilon)\zeta\left(
\varepsilon s\right)  +\varepsilon\zeta^{\prime}\left(  \varepsilon s\right)
\right) \label{f 1 der}%
\end{equation}
In a similar way we can obtain the second derivative of $f_{k}(s,\varepsilon
)$. Introducing denotations:%
\[
q_{k}(s,\varepsilon)\equiv p_{k}(s,\varepsilon)^{2}+p_{k}^{(1)}(s,\varepsilon)
\]%
\[
q_{k}^{(1)}(s,\varepsilon)\equiv2\varepsilon p_{k}(s,\varepsilon)
\]
we get:%
\begin{equation}
f_{k}^{(2)}(s,\varepsilon)\equiv\frac{\partial^{2}}{\partial s^{2}}%
f_{k}(s,\varepsilon)=g_{k}(s,\varepsilon)\left(  q_{k}(s,\varepsilon
)\zeta\left(  \varepsilon s\right)  +q_{k}^{(1)}(s,\varepsilon)\zeta^{\prime
}\left(  \varepsilon s\right)  +\varepsilon^{2}\zeta^{^{\prime\prime}}\left(
\varepsilon s\right)  \right) \label{f 2 der}%
\end{equation}

Inserting (\ref{f 1 der}) and (\ref{f 2 der}) into (\ref{omega 1 der}) and
(\ref{omega 2 der}) we get:%
\begin{equation}
\frac{\partial}{\partial s}\omega_{k}(s,\varepsilon)=p_{k}(s,\varepsilon
)-\frac{\pi\varepsilon}{2}\tan\left(  \frac{\pi s\varepsilon}{2}\right)
+\varepsilon\frac{\zeta^{^{\prime}}\left(  \varepsilon s\right)  }%
{\zeta\left(  \varepsilon s\right)  }\label{omega 1 der new}%
\end{equation}%
\begin{equation}
\frac{\partial^{2}}{\partial s^{2}}\omega_{k}(s,\varepsilon)=q_{k}%
(s,\varepsilon)-\left(  \frac{\pi\varepsilon}{2}\tan\left(  \frac{\pi
s\varepsilon}{2}\right)  \right)  ^{2}+\varepsilon^{2}\left(  \frac
{\zeta^{^{^{\prime\prime}}}\left(  \varepsilon s\right)  }{\zeta\left(
\varepsilon s\right)  }-\left(  \frac{\zeta^{^{\prime}}\left(  \varepsilon
s\right)  }{\zeta\left(  \varepsilon s\right)  }\right)  ^{2}\right)
\label{omega 2 der new}%
\end{equation}

Having explicitly calculated derivatives of the integrand we are ready to
apply the saddle point method (\ref{saddle point method}). First we have to
find the location of saddle points. Equating (\ref{f 1 der}) or
(\ref{omega 1 der new}) to zero gives:%
\begin{equation}
p_{k}(s,\varepsilon)\zeta\left(  \varepsilon s\right)  +\varepsilon
\zeta^{\prime}\left(  \varepsilon s\right)  =0\label{Saddle equation}%
\end{equation}
Note that for variable $s$\ having large imaginary part we have (cf. e.g.
\cite{Flajolet Vepstas}, formula (20)):%
\begin{align}
\zeta\left(  s\right)   & \sim1\label{Zeta limits}\\
\zeta^{^{\prime}}\left(  s\right)   & \sim0\nonumber\\
\zeta^{^{^{\prime\prime}}}\left(  s\right)   & \sim0\nonumber
\end{align}

The above approximations seem very radical and illegitimate, because the zeta
function seems to disappear from the reasoning at this stage. Nevertheless,
they are satisfied with accuracy to many significant digits along the
integration path which, let us recall, can be shifted arbitrarily far to the
right. After all, the zeta is present there, at least by its functional
equation (\ref{Functional equation}), which has been used above.

Therefore, instead of (\ref{Saddle equation}), we simply get:%
\begin{equation}
3+2k+s\varepsilon\left(  2\ln\frac{2\pi}{s\varepsilon}\pm i\pi\right)
=0\label{Saddle equation 1}%
\end{equation}
(One has to be careful with logarithms of complex arguments so as not to
ignore a case and therefore not miss a solution.) Equation
(\ref{Saddle equation 1}) may be solved explicitly with respect to $s$ giving
the complex location of $k^{\text{th}}$ saddle point (for a given small
parameter $\varepsilon$):%
\begin{equation}
s_{k}=\frac{k+\frac{3}{2}}{\varepsilon W\left(  \pm\frac{k+\frac{3}{2}}{2\pi
i}\right)  }\label{Saddles}%
\end{equation}
where $W$ is the Lambert function satisfying transcendental functional
equation:%
\[
s=W(s)e^{W(s)}%
\]
Incidentally, formula (\ref{Saddles}) resembles approximate formula for the
imaginary parts $y_{n}$ of complex zeta zeros found by Andr\'{e} LeClair (see
\cite{LeClair}, formula (22)):%
\[
y_{n}=\frac{n-\frac{11}{8}}{W\left(  \frac{n-\frac{11}{8}}{e}\right)  }%
\]
From (\ref{Saddles}) it is evident that distribution of saddle points on the
complex plane scales as the inverse of parameter $\varepsilon$.%
\begin{figure}[ptb]%
\centering
\includegraphics[height=0.5\textheight]
{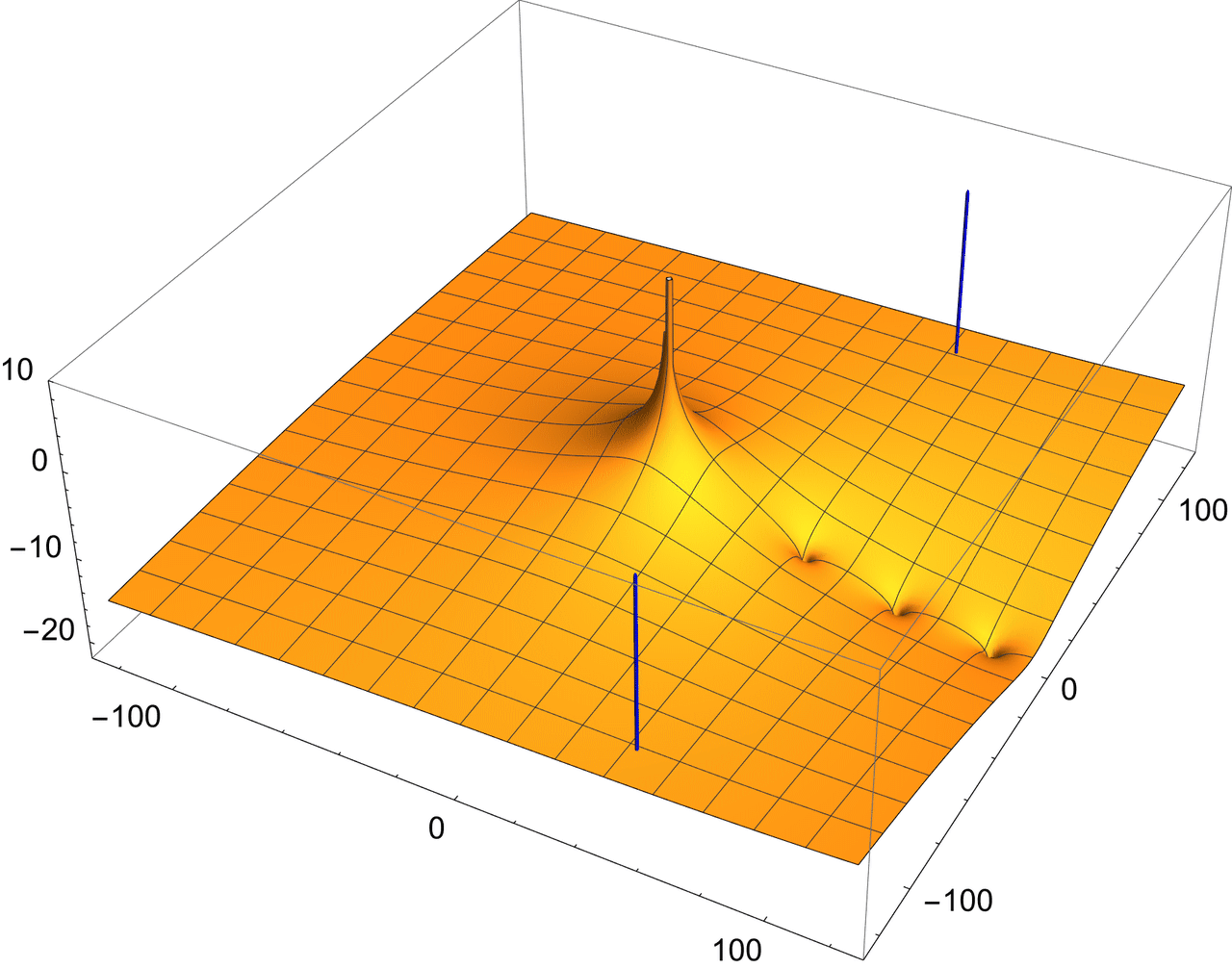}%
\caption{The logarithm of the absolute value of the integrand (\ref{integrand}%
) for $k=2$ and $\varepsilon=2^{-4}$. Positions of saddle points is marked by
vertical lines. The saddle nature of these points is practically invisible due
to the scale of the figure. Also, three singularities for $s=0,-1,-2$ merged
into single peak. Better visualisation is presented on the next Figure 5.}%
\end{figure}
\begin{figure}[ptb]%
\centering
\includegraphics[height=0.85\textheight]
{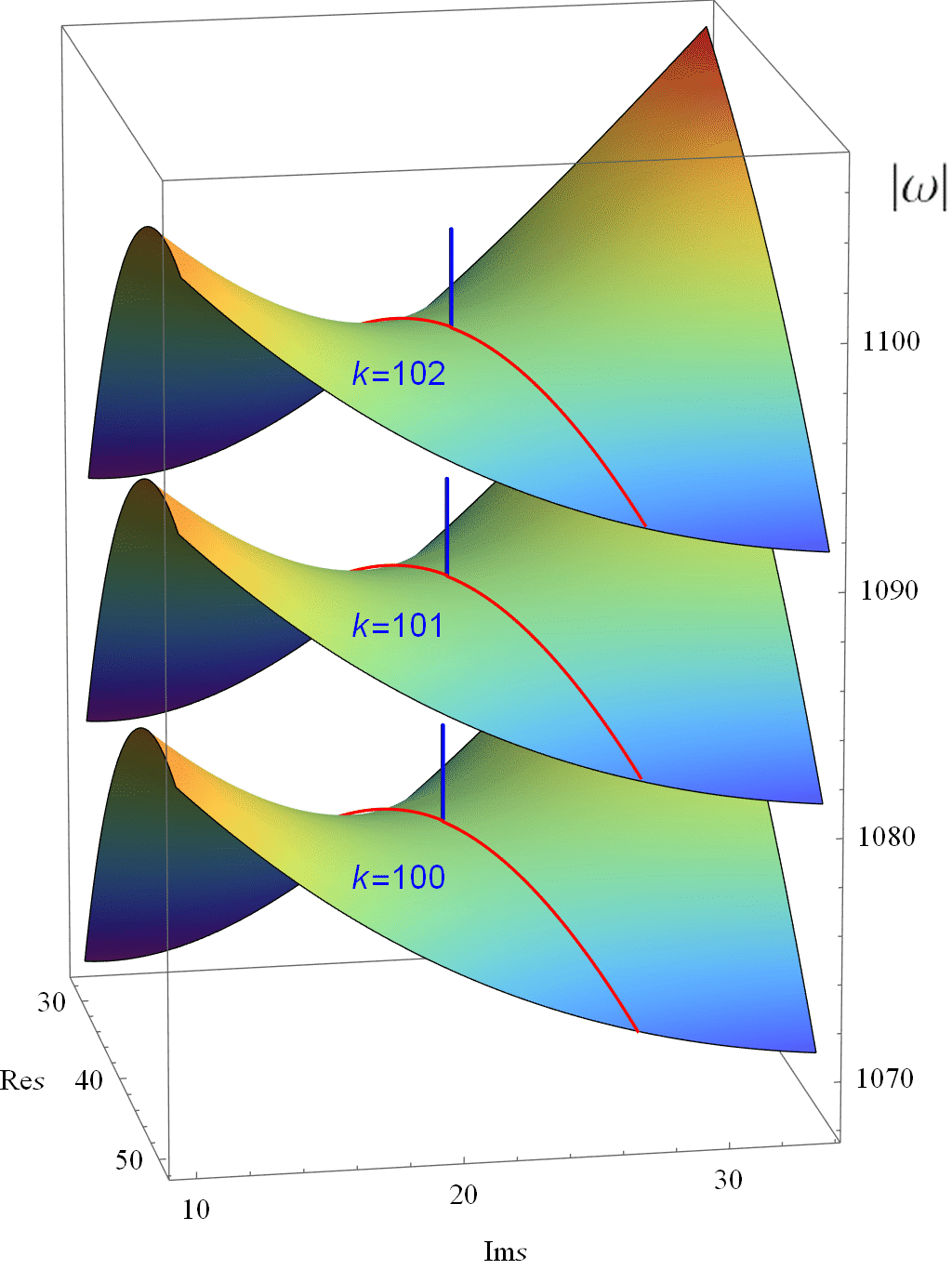}%
\caption{Typical family of fragments of absolute values of the function
$\omega_{k}(s,\varepsilon)$ (\ref{ln integrand}) in the vicinity of saddle
points for $k=100,101$ and $102$. Blue vertical segments mark the position of
the saddles. The red lines are the curves of the steepest descent.}%
\end{figure}

\section{Completion of computations}

Having calculated the second derivative of the integrand and the positions of
the stationary points, we can finally use the theorem
(\ref{saddle point method}) and provide an asymptotic expression for $a_{k}$:%
\begin{equation}
a_{k}(\varepsilon)\simeq-\operatorname{Re}\left(  \frac{k!}{\pi i}\sqrt
{\frac{2\pi}{-\frac{\partial^{2}}{\partial s^{2}}\omega_{k}(s_{k}%
,\varepsilon)}}f_{k}(s_{k},\varepsilon)\right) \label{a_k asymptotics}%
\end{equation}
To get the sought asymptotic formula for $\gamma_{n}$ coefficients, all that
remains is to insert (\ref{a_k asymptotics}) into the general expression
(\ref{gamma gen}) and make some elementary approximations. (As always, Mathematica
procedures such as Limit, Series, etc. save a lot of time and effort while
ensuring that the results are error free.) In particular:%
\begin{equation}
\frac{\partial^{2}}{\partial s^{2}}\omega_{k}(s_{k},\varepsilon)\simeq
\frac{3+2k}{2s^{2}}+\frac{\varepsilon}{s}-\frac{\pi^{2}\varepsilon^{2}}%
{4}\left(  1+\tan\left(  \frac{\pi s\varepsilon}{2}\right)  ^{2}\right)
\label{omega 2 der approx}%
\end{equation}
Using (\ref{Zeta limits}) we can also put:%
\begin{equation}
f_{k}(s,\varepsilon)\simeq g_{k}(s,\varepsilon)\label{f approx}%
\end{equation}
Remembering that for large imaginary part of $s$%
\begin{align*}
\Gamma(s)  & \simeq\sqrt{\frac{2\pi}{s}}e^{-s}s^{s}\\
\cos(s)  & \simeq\frac{e^{-is}}{2}%
\end{align*}
we have:%
\begin{equation}
g_{k}(s,\varepsilon)\simeq\frac{1}{2}\pi^{\frac{1}{2}-s\varepsilon}%
s^{-k-2}\left(  2s-k-k^{2}\right)  \frac{\Gamma(\frac{s\varepsilon}{2}%
)}{\Gamma(\frac{1-s\varepsilon}{2})}\label{g approx}%
\end{equation}
It is clear that, since finally $\varepsilon$ tends to zero, it is sufficient
to take only the first term in (\ref{gamma gen})%
\begin{equation}
\gamma_{n}\simeq\frac{a_{k}(\varepsilon)}{\varepsilon^{n}}\quad\quad
\varepsilon\rightarrow0\label{gamma approx}%
\end{equation}
Inserting to (\ref{gamma approx}) expression for $a_{k}(\varepsilon)$
(\ref{a_k asymptotics}) together with (\ref{omega 2 der approx}),
(\ref{f approx}) and (\ref{g approx}) we finally get:%
\[
\gamma_{n}\sim\sqrt{\frac{2}{\pi}}n!\operatorname{Re}\frac{\Gamma\left(
s_{n}\right)  e^{-cs_{n}}}{\left(  s_{n}\right)  ^{n}\sqrt{n+s_{n}+\frac{3}%
{2}}}%
\]
(In fact, there is always a pair of mutually conjugate saddles but
contributions due to their imaginary parts cancels.)

It is probably quite astonishing that after making so many approximations the
final formula for $\gamma_{n}$ works so well as computer experiments show
convincingly. As expected, in this formula there is no longer the auxiliary
parameter $\varepsilon$, which fulfilled its important but temporary role in
numerical computations (with the help of formula (\ref{gamma gen})), and finally
simply get shortened.

\section{Summary of results}

Let's collect the final results. Let $c$ be a complex constant:%
\[
c=\log(2\pi)+\frac{\pi}{2}i=\log(2\pi i)
\]

Now asymptotics of Stieltjes constants when $n\rightarrow\infty$ (in practice
it suffices that $n\gg0$) is:%
\begin{equation}
\gamma_{n}\sim\sqrt{\frac{2}{\pi}}n!\operatorname{Re}\frac{\Gamma\left(
s_{n}\right)  e^{-cs_{n}}}{\left(  s_{n}\right)  ^{n}\sqrt{n+s_{n}+\frac{3}%
{2}}}\label{asymptotics 0}%
\end{equation}
where complex saddle points are (note that now there is no $\varepsilon$ which
get shortened):%
\begin{equation}
s_{n}=\frac{n+\frac{3}{2}}{W\left(  \pm\frac{n+\frac{3}{2}}{2\pi i}\right)
}\label{Saddles 0}%
\end{equation}
The $250$ initial values of the complex saddles (\ref{Saddles 0}) are shown in
the Figure 6. Very good agreement of approximated values calculated using
(\ref{asymptotics 0}) with actual values of $\gamma_{n}$ is shown in Figure 7.%
\begin{figure}[ptb]%
\centering
\includegraphics[height=0.5\textheight]
{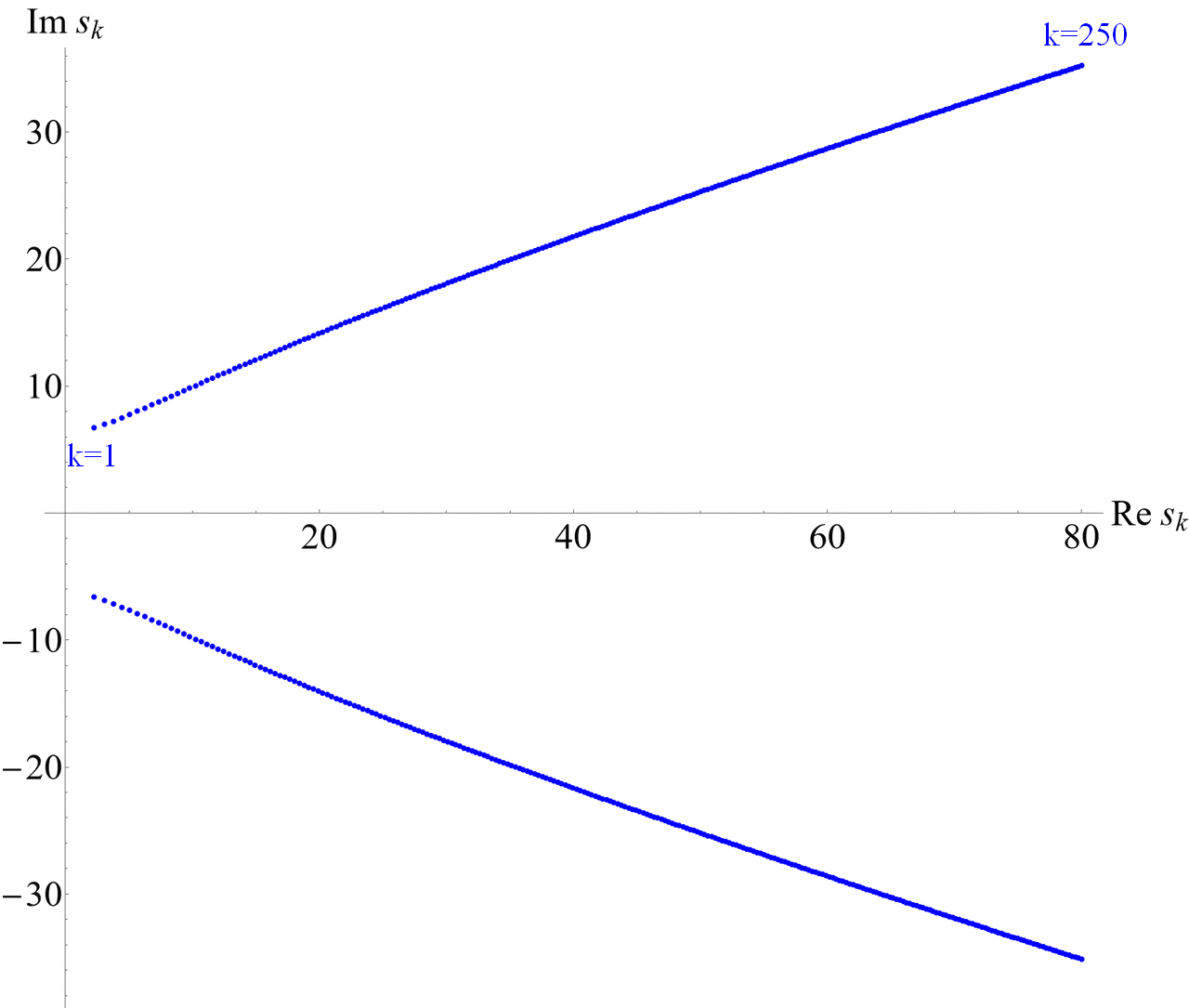}%
\caption{Distribution of 250 initial saddle points on the complex plane. There
are two symmetrical branches, the lower one is the complex conjugate of the
upper one.}%
\end{figure}
\begin{figure}[ptb]%
\centering
\includegraphics[height=0.4\textheight]
{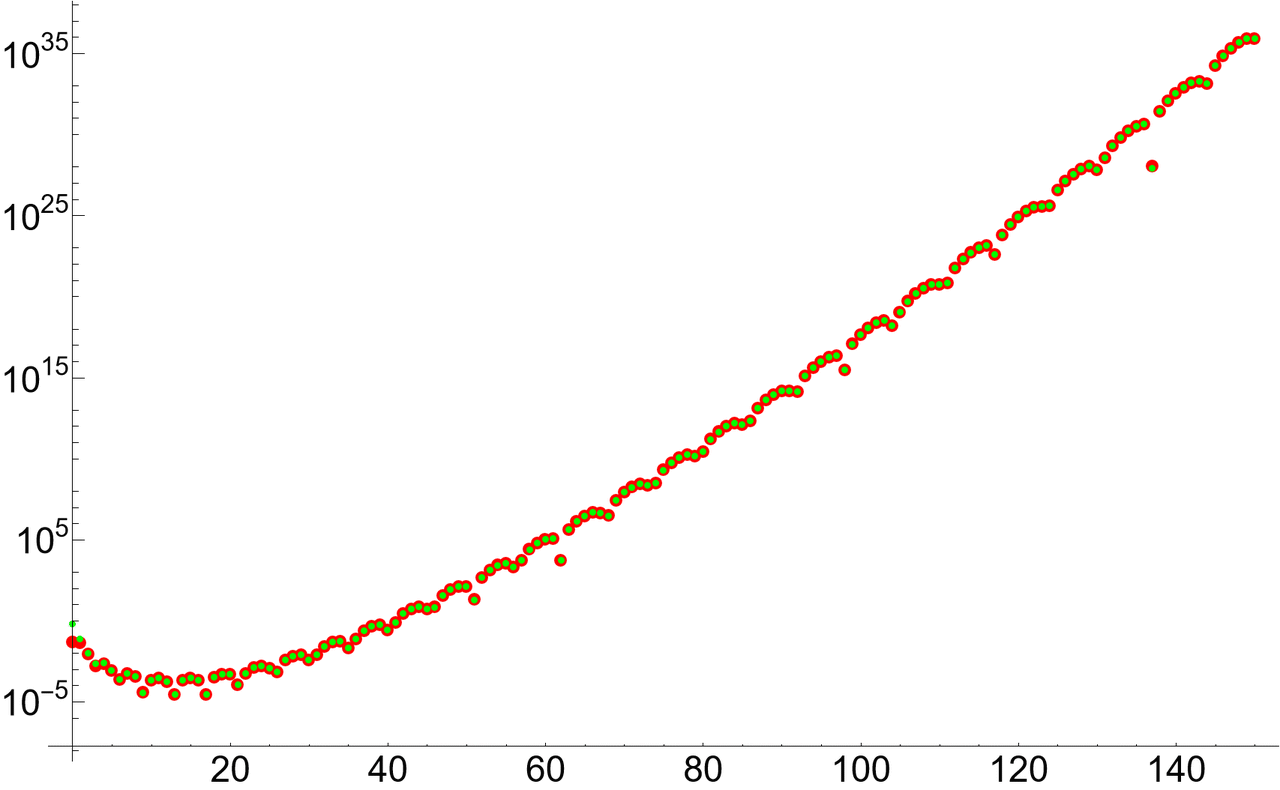}%
\caption{Comparison of absolute values of actual Stieltjes coefficients (green
dots) with those calculated from asymptotic formula (\ref{asymptotics 0}) (red
dots) shows good agreement (except $\gamma_{0}$), even for that "unruly" value
$n=137$.}%
\end{figure}
\begin{figure}[ptb]%
\centering
\includegraphics[height=0.4\textheight]
{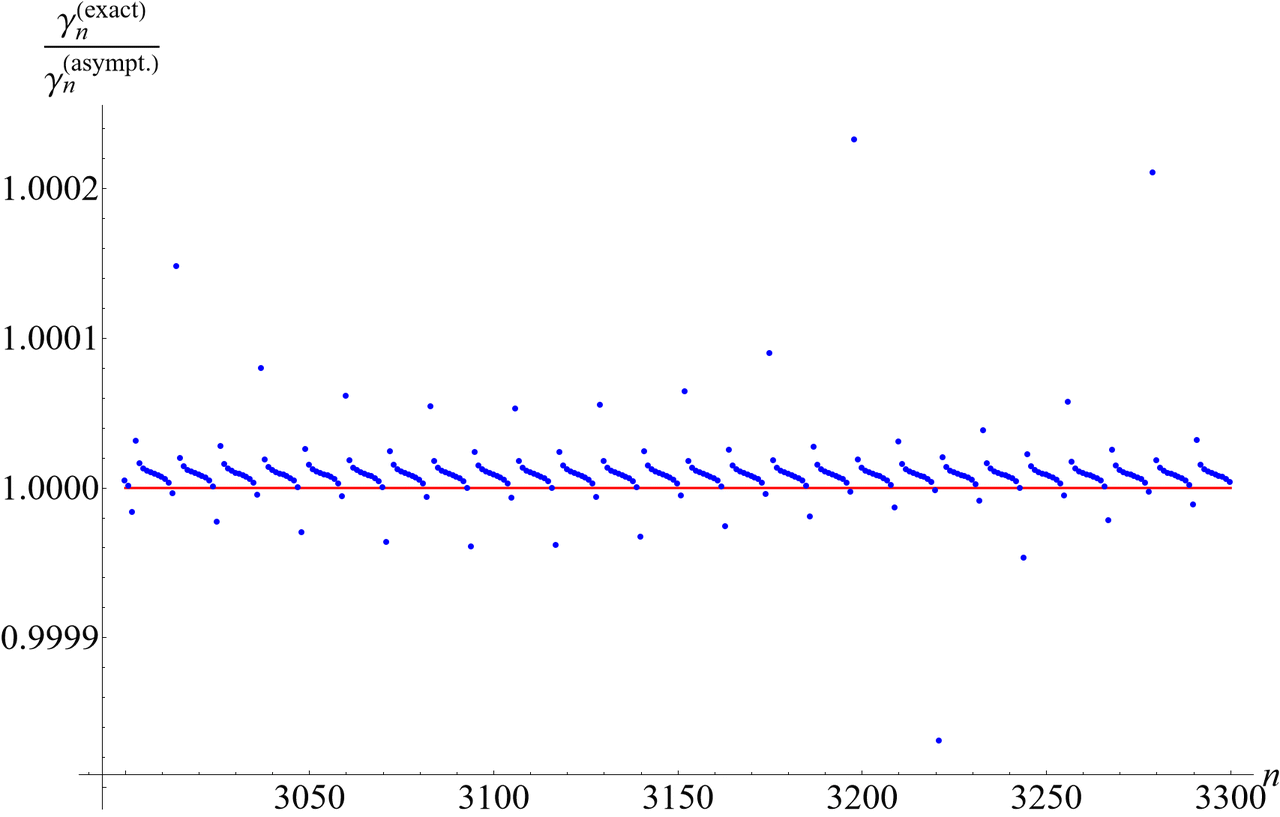}%
\caption{Unfortunately, the good impression after looking at Figure 7
diminishes a bit when we look at the graph of the ratio of the exact
$\gamma_{n}$ values to the asymptotic values (\ref{asymptotics 0}) for example
in the range of $n=3000-3300$. Although this ratio is very close to one, with
an accuracy generally better than $10^{-4}$, there are distinct, periodic
structures: points are arranged on certain curves resembling the family of
cotangent functions. But it is in these structures that the essence of
Riemann's zeta is contained, including the Riemann hypothesis, i.e. what was
rejected when the approximations (\ref{Zeta limits}) were made. Also note,
which is somewhat surprising, that the blue points lie slightly above the red
straight line that represents the value of one. But, as the saying goes, when
one door shuts, another one opens. And indeed: this result opens up a whole
new field for very fruitful research on the Stieltjes constants, which will be
the subject of the next publication.}%
\end{figure}

\section{Application: Signs of $\gamma_{n}$}

As a by-product of these intricate computations, we can get a compact
expression for the signs of the Stieltjes constants. Formula
(\ref{asymptotics 0}) hides the characteristic behavior of Stieltjes constants
when $n$ grows, that is large and growing oscillations with diminishing
frequency superimposed on the strongly growing trend. This behavior may be
demonstrated as follows. Recall higher order Stirling formula for $\Gamma(x)
$:%
\[
\Gamma(x)\simeq\frac{1}{6}\sqrt{\frac{\pi}{2}}e^{-x}x^{x-\frac{3}{2}}(12x+1)
\]
Applying it to $\Gamma\left(  s_{n}\right)  $ in (\ref{asymptotics 0}) we get:%
\begin{equation}
\gamma_{n}\simeq2n!\operatorname{Re}\frac{\left(  s_{n}\right)  ^{s_{n}%
-n-\frac{3}{2}}(s_{n}+\frac{1}{12})e^{-(c+1)s_{n}}}{\sqrt{n+s_{n}+\frac{3}{2}%
}}\label{asymptotics 2}%
\end{equation}
For $n\gg1$ fractions $\frac{1}{12}$ and $\frac{3}{2}$ under the square root
may be neglected since $s_{n}$ grows fast with $n$:%
\begin{align*}
\gamma_{n}  & \simeq2n!\operatorname{Re}\frac{\left(  s_{n}\right)
^{s_{n}-n-\frac{1}{2}}e^{-(c+1)s_{n}}}{\sqrt{n+s_{n}}}=\\
& =2n!\operatorname{Re}\frac{\exp\left[  (s_{n}-n-\frac{1}{2})\ln\left(
s_{n}\right)  \right]  e^{-(c+1)s_{n}}}{\exp\left[  \frac{1}{2}\ln\left(
n+s_{n}\right)  \right]  }=\\
& =2n!\operatorname{Re}\exp\left[  (s_{n}-n-\frac{1}{2})\ln\left(
s_{n}\right)  -\frac{1}{2}\ln\left(  n+s_{n}\right)  -(c+1)s_{n}\right]
\end{align*}

Applying once again Stirling formula to $n!$ we have:%
\begin{equation}
\gamma_{n}\simeq\sqrt{8\pi}\operatorname{Re}\exp\left[  \frac{1}{2}%
\ln(n)+n\left(  \ln(n)-1\right)  +(s_{n}-n-\frac{1}{2})\ln\left(
s_{n}\right)  -\frac{1}{2}\ln\left(  n+s_{n}\right)  -(c+1)s_{n}\right]
\label{gamma approx 1}%
\end{equation}
Introducing finally complex "phase" as:%
\begin{equation}
\varphi_{n}\equiv\frac{1}{2}\ln(8\pi)-n+(n+\frac{1}{2})\ln(n)+(s_{n}%
-n-\frac{1}{2})\ln\left(  s_{n}\right)  -\frac{1}{2}\ln\left(  n+s_{n}\right)
-(c+1)s_{n}\label{phase}%
\end{equation}
we get particularly simple expression:%
\begin{equation}
\gamma_{n}\simeq\operatorname{Re}\left[  e^{\varphi_{n}}\right]
=e^{\operatorname{Re}\varphi_{n}}\cos\left(  \operatorname{Im}\varphi
_{n}\right) \label{gamma approx 2}%
\end{equation}%
\begin{figure}[ptb]%
\centering
\includegraphics[height=0.7\textheight]
{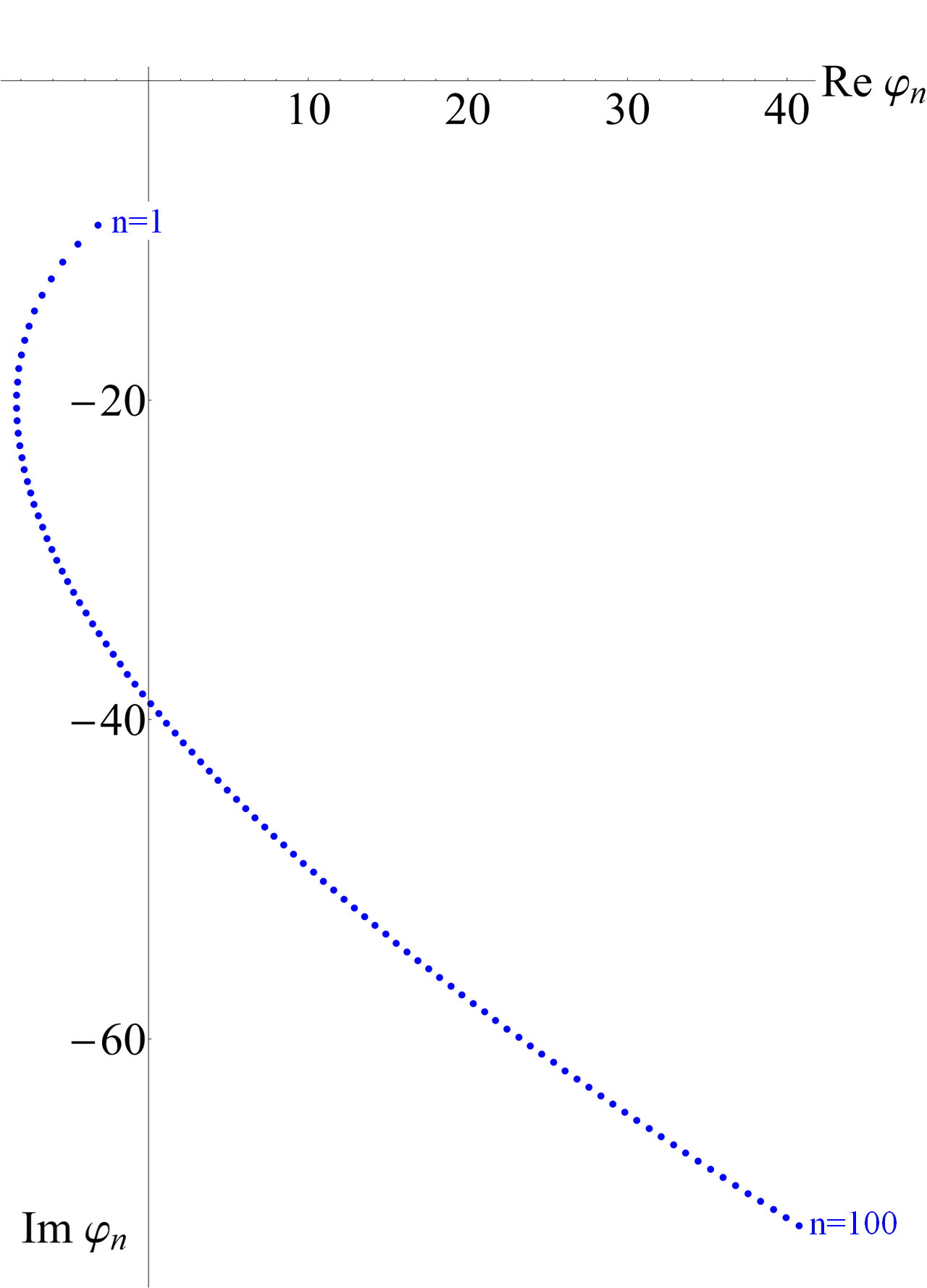}%
\caption{Distribution of complex values of phase $\varphi_{n}$ given by
(\ref{phase}) for $n=1,2,...,100$. It is clear that thay lie along certain
smooth curve. It is also obvious that the exponent of the real part of the
phase (\ref{phase}) controls the rapid growth of $\gamma_{n}$, while the
cosine of the imaginary part of the phase is responsible for the oscillations
of decreasing frequency.}%
\end{figure}
Formula (\ref{gamma approx 2}) gives almost as good approximation as
(\ref{asymptotics 0}) but it shows in a manifest way mentioned above basic
properties of $\gamma_{n}$ (trend and oscillations). It is then clear that the
statement quoted at the beginning that "Stieltjes constants [...] change signs
in a complex pattern" \cite{Wikipedia} is not true. In particular, one can
quickly calculate sign of $\gamma_{n}$, even for extremely high $n$,\ since it
is obviously equal to the sign of $\cos\left(  \operatorname{Im}\varphi
_{n}\right)  $ and the phase (\ref{phase}) can be computed effectively for $n$
at least up to $10^{1,000,000}$. (See \cite{OEIS} for extensive computations
of signs of Stieltjes constants using the above formulas.) For example:

\bigskip%

\begin{tabular}
[c]{ll}%
$n$ & sign of $\gamma_{n}$\\
$10^{10}$ & $+1$\\
$10^{100}$ & $+1$\\
$10^{1000}$ & $+1$\\
$10^{10\,000}$ & $-1$\\
$10^{100\,000}$ & $-1$\\
$10^{1,000\,000}$ & $+1$%
\end{tabular}

\section{Appendix - samples of Mathematica notebooks}

As mentioned in the main text, Wolfram's Mathematica \cite{Wolfram
Mathematica} made very tedious and convoluted computations much easier and
ensured that there were no mistakes in them. This program was used very
intensively -- for symbolic transformations and in terms of its enormous
purely numerical capabilities and finally for its rich graphical presentations
of the obtained results. Figure 9 is an example of how well Mathematica is
doing to check that the contour integral (\ref{alpha_k integral 0}) is indeed
equal to the binomial alternating sum (\ref{a_k zeta}). I cannot imagine how
to verify this fact with such high precision without computer support. Another
example: Figure 10 shows how Mathematica solves the transcendental equations
(\ref{Saddle equation 1}).%
\begin{figure}[ptb]%
\centering
\includegraphics[height=0.8\textheight]
{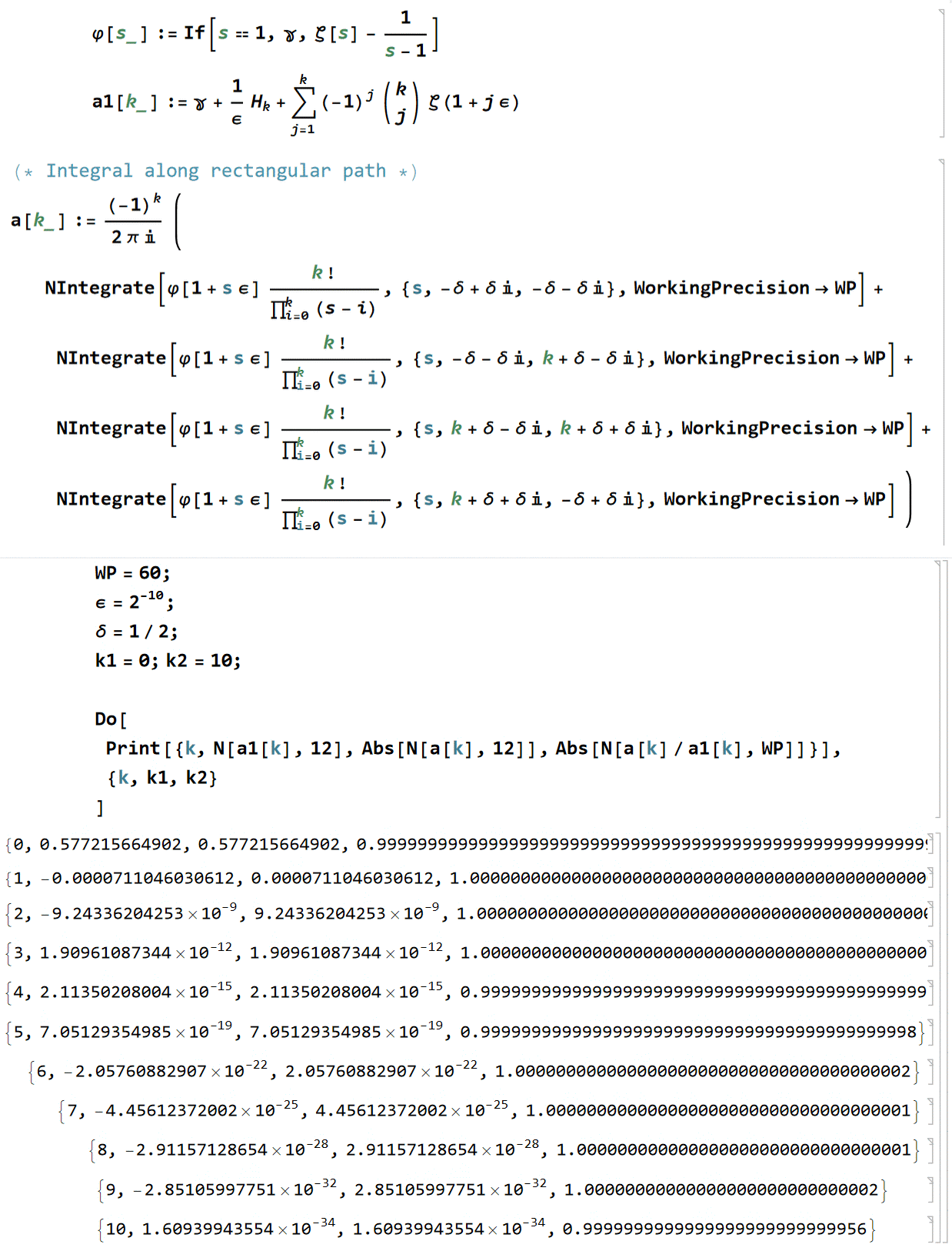}%
\caption{Checking the N\o rlund-Rice formula (\ref{Noerlund-Rice}) using
Mathematica.}%
\end{figure}
\begin{figure}[ptb]%
\centering
\includegraphics[height=0.35\textheight]
{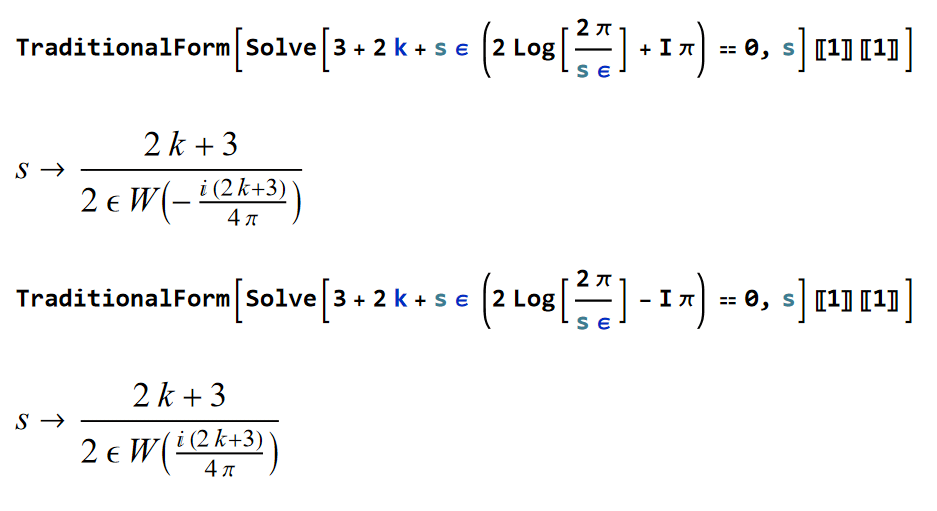}%
\caption{Illustration of the \textbf{Solve} procedure capabilities applied to
equations (\ref{Saddle equation 1}).}%
\end{figure}

\end{document}